\documentclass[12 pt]{amsart}
\usepackage{amsmath,amsthm,amsfonts,amssymb,amscd,amsrefs}
\usepackage[top=3cm, bottom=3cm,left=2cm, right=2cm]{geometry}
\usepackage{centernot}
\usepackage[linkcolor=red,citecolor=red,colorlinks=true]{hyperref}
\usepackage{mathrsfs}
\usepackage{mathtools}
\usepackage{etoolbox}
\usepackage{scalerel,stackengine}
\stackMath
\newcommand\widecheck[1]{%
	\savestack{\tmpbox}{\stretchto{%
			\scaleto{%
				\scalerel*[\widthof{\ensuremath{#1}}]{\kern-.6pt\bigwedge\kern-.6pt}%
				{\rule[-\textheight/2]{1ex}{\textheight}}
			}{\textheight}%
			
		}{0.5ex}}%
	\stackon[1pt]{#1}{\scalebox{-1}{\tmpbox}}%
}
\theoremstyle{definition}
\newtheorem{defn}{Definition}[section]
\newtheorem{qn}[defn]{Question}
\newtheorem{eg}[defn]{Example}

\theoremstyle{remark}

\newtheorem{rmk}[defn]{Remark}
\theoremstyle{plain}
\newtheorem{thm}[defn]{Theorem}
\newtheorem{lem}[defn]{Lemma}
\newtheorem{cor}[defn]{Corollary}
\newtheorem{prop}[defn]{Proposition}

\DeclareMathOperator*{\spn}{span}

\DeclareMathOperator*{\conv}{conv}

\numberwithin{equation}{section}
\makeatletter 

\newcommand{\Rmnum}[1]{\expandafter\@slowromancap\romannumeral #1@}

\makeatother
\makeatletter
\def\imod#1{\allowbreak\mkern10mu({\operator@font mod}\,\,#1)}
\makeatother
\title{Differentiability of the operator norm on $\ell_p$ spaces}
\author[Sreejith Siju]{Sreejith Siju}
\address{Kerala School of Mathematics,
Kunnamangalam PO,
Kozhikode, Kerala, 673571, India}
\email{sreejithsiju5@gmail.com}

\begin{document}
\maketitle
\begin{abstract}

    In this paper, we present a characterization of strong subdifferentiability of the norm of bounded linear operators on $\ell_p$ spaces, $1\leq p<\infty$.
Furthermore, we prove that the set of all bounded linear operators in ${B}(\ell_p, \ell_q)$ for which the norm of ${B}(\ell_p, \ell_q)$ is strongly subdifferentiable is dense in ${B}(\ell_p, \ell_q)$. 
Additionally, 
   we present a characterization of Fr\'{e}chet differentiability of the norm of bounded linear operators from $\ell_p$ to $\ell_q$, where $1 < p, q < \infty$. 
Applying this result, we will show that the Fr\'{e}chet differentiability and the Gateaux differentiability of the norm of bounded linear operators on $\ell_p$ spaces coincide, extending a known theorem regarding the operator norm on Hilbert spaces. 
 
\end{abstract}

\section{Introduction}

 The differentiability of the norm is a very useful tool in understanding the geometry of the underlying Banach space. 
 Various notions of differentiability of the norm of Banach spaces include Fr\'{e}chet, Gateaux, and strong subdifferentiability (referred to as $SSD$) (see Definitions \ref{ssddef} and \ref{fred}).  It is well known that a point $x$ in a Banach space $X$ is a smooth point if and only if the norm of $X$ is Gateaux differentiable at $x$. We employ the terms Gateaux differentiability and smoothness interchangeably. 
The Fr\'{e}chet differentiability of the norm of a Banach space at point $x$ is equivalent to the existence of the unique support functional at $x$ which is strongly exposed by $x$. 
 Strong subdifferentiability provides a non-smooth extension of Fr\'{e}chet differentiability, relating Fr\'{e}chet and Gateaux differentiability. That is, the norm of a Banach space is Fr\'{e}chet differentiable at a point $x$ if and only if it is Gateaux differentiable and strongly subdifferentiable at $x$ \cite{MR1216708}. 
 The monographs \cite{MR0461094} and \cite{MR1211634} provide a detailed discussion of the concepts Fr\'{e}chet and Gateaux differentiability in Banach spaces. 
For a detailed study of strong subdifferentiability of the norm, we recommend referring to the discussions found in \cite{MR1216708, MR1364490, MR1911087}, and the references therein. 

This article primarily investigates the strong subdifferentiability of operator norm over the sequence spaces $\ell_p$, where $1<p<\infty$. Additionally, we study the Fr\'{e}chet and Gateaux differentiability within the realm of operators acting on $\ell_p$ spaces. 

The strong subdifferentiability of Banach space norms has been studied extensively in the literature. We refer the reader \cite{MR1216708} for the basic theory of strong subdifferentiability. Strong subdifferentiability has been also studied in the context of dual spaces, operator norm, and $JB^*$-triples (see \cite{MR1373537, MR4497173, MR2031171, MR1911087}). The connection between strong subdifferentiability and Bishop-Phelps Bollobas properties has been explored in \cite{MR4049872}. To see connections between strong subdifferentiability and other notions in Banach space theory, we refer \cite{MR2081949, MR1397934, MR4653431, MR0806469}. 

   An interesting characterization of the strong subdifferentiability of the dual norm of a Banach space $X$ at a bounded linear functional $f$ has been obtained in \cite{MR1911087} by G. Godefroy, V. Indumathi, and F. Lust-Piquard. The characterizing condition for strong subdifferentiability is established in terms of the set of norm-attaining points of $f$ (see Theorem \ref{indu}). 
 Moreover, in \cite{MR1373537}, M. Contreras, R. Paya, and W. Werner characterized the strong subdifferentiability of the operator norm of Hilbert spaces (see Theorem \ref{hilb}), and showed that the collection of all strongly subdifferentiable points is dense in the space of bounded linear operators on a Hilbert space. Subsequently, J. Guerrero and A. Palacios broadened this characterization of strong subdifferentiability to encompass $JB^*$-triples \cite{MR2031171}.  

 In Theorem \ref{hil}, we will show that the characterizing conditions of strong subdifferentiability obtained in \cite{MR1911087} (see Theorem \ref{indu}) and \cite{MR1373537}(see Theorem \ref{hilb}) are essentially the same if we view from the perspective of maximizing sequence (see Definition \ref{maxim}). This equivalence gives rise to two natural questions in this context.

\begin{qn}\label{qn1}
     Is it possible to extend Theorem \ref{hil} to bounded linear operators on Banach spaces? 
 \end{qn}
 \begin{qn}\label{qn2}
     Is the set of strongly subdifferentiable points of ${B}(X)$, the space of bounded linear operators on $X$, dense in ${B}(X)$?
 \end{qn}
 
In Section 2, we will show that Question \ref{qn1} has an affirmative answer in the case of $\ell_p$ spaces, $1<p<\infty$. We obtain a characterization for the strong subdifferentiability of operator norm on $\ell_p$ spaces in terms of the maximizing sequence of an operator, as observed in the case of dual norm and operator norm on Hilbert space. We also prove that such operators attain their norm. Our idea here is to first formulate a condition that is equivalent to the characterization of strong subdifferentiability of operator norm Hilbert spaces found in \cite{MR1373537} (see Theorem \ref{hil}). Then we will prove the equivalent condition characterizes the strong subdifferentiability of operator norm on $\ell_p$ spaces. We will also address the borderline cases $p=1,\infty$ towards the end of this section. 

In Section 3, we will give an affirmative answer to Question \ref{qn2} in the case of $\ell_p$ spaces, $1\leq p <\infty$.  We will exhibit that, similar to the case of operator norm on Hilbert spaces, the set of all strongly subdifferentiable points is norm dense in the space of bounded linear operators on $\ell_p$ spaces. This will also extend, in $\ell_p$ space situation, a result of J. Lindenstrauss \cite{MR0160094} regarding norm-attaining operators to the class of strongly subdifferentiable points of the space of bounded linear operator. 

In \cite{MR1073855}, F. Kittaneh and R. Younis presented a characterization of smooth points in the space of bounded linear operators on a Hilbert space, ${B}(H)$, using the essential norm of an operator. Subsequently, W. Deeb and R. Khalil extended this characterization to include operators on $\ell_p$ spaces \cite{MR1151547}. In another work \cite{MR1149980} by W. Werner and K. F. Taylor, it was exhibited that for a bounded linear operator $T$ on a Hilbert space $H$, the assumption made by F. Kittaneh and R. Younis regarding the essential norm of $T$ goes beyond characterizing smoothness and serves as a characterization for the Fr\'{e}chet differentiability of the norm at $T$. Thereby, proving that Fr\'{e}chet and Gateaux differentiability coincide in ${B}(H)$. Furthermore, it has been independently proved in \cite{MR519008} that the equivalence between Fr\'{e}chet and Gateaux differentiability holds in ${B}(H)$. So, it is natural to inquire whether the coincidence exists between the Fr\'{e}chet and Gateaux differentiability of the norm of the space of bounded linear operators on $\ell_p$ spaces.

Towards the conclusion of section 3, we establish, as a consequence of the results on strong subdifferentiability, the equivalence between Fr\'{e}chet and Gateaux differentiability of the operator norm on $\ell_p$ spaces. Further, we will show that the assumption made by F. Kittaneh and R. Younis regarding the essential norm of $T$ serves as a characterization for the Fr\'{e}chet differentiability of operator norm on $\ell_p$ spaces (see Proposition \ref{equiv}).

\textit{Notations:} In this article, we will use $X$ and $Y$ to represent infinite dimensional Banach spaces, $H$ to denote an infinite dimensional Hilbert space, and the scalar fields are either real or complex, denoted by $\mathbb{F}$. 
When referring to a subspace of a Banach space, we will always assume it to be a closed subspace. 
The closed unit ball and the unit sphere of a Banach space $X$ will be denoted by $B_X$ and $S_X$, respectively.
We consider every Banach space $X$ as a subspace of its bidual $X^{**}$ through the canonical embedding. 
For a subset $Y$ of a Banach space $X$, the annihilator of $Y$ in the dual space $X^*$ is denoted by $Y^\perp$, defined as $Y^\perp=\{f\in X^*: f(y)=0 \quad \forall y\in Y\}$. 
The convex hull of a set $S$ is the smallest convex set containing $S$, denoted by $\operatorname{conv}(S)$.
The closure of $\operatorname{conv}(S)$ is  represented by $\overline{\operatorname{conv}}(S)$. For a Banach space $X$, the extreme points of $B_X$ will be denoted by $\operatorname{ext} B_{X}$. For a bounded subset $S$ of $X$, the diameter of $S$ will be denoted by $\operatorname{diam}S$.

Let $X$ and $Y$ be Banach spaces. We denote by ${B}(X, Y)$, the space of all bounded linear operators from $X$ to $Y$ endowed with the usual operator norm. The space of all compact operators from $X$ to $Y$ is denoted by ${K}(X, Y)$. We denote $B(X\times Y)$ to be the space of all bounded bilinear functional from $X\times Y$ to $\mathbb{F}$. When $X=Y$, we abbreviate ${B}(X, Y)$ as ${B}(X)$, and ${K}(X, Y)$ as ${K}(X)$. The essential norm of a bounded linear operator $T\in {B}(X, Y)$ is defined as the distance from $T$ to the space of compact operators, denoted by $\|T\|_e$. For an operator $T\in {B}(X, Y)$, we define the set $M_T$ to be the set $M_T=\{x\in S_X: \|Tx\|=\|T\|\}$. We call the set $M_T$, the norm-attaining set of $T$.
For, $1\leq p<\infty$,  the direct sum of $X$ and $Y$ endowed with the norm $\|(x, y)\| = (\|x\|^p+\|y\|^p)^{\frac{1}{p}}$ is denoted by $X\oplus_p Y$. 
We denote by $X\hat{\otimes}_\pi Y$, and $X\check{\otimes}_\varepsilon Y$, the projective and injective tensor product of $X$ and $Y$, respectively.

\section{Strong subdifferentiability of operator norm}

In this section, we will characterize strong subdifferentiability of the norm of  ${B}(\ell_p,\ell_q)$ at an operator  $T\in {B}(\ell_p,\ell_q)$ in terms of maximizing sequence of $T$, where $1<p,q<\infty$. To begin, we will review some definitions.
We will first recall the definition of strong subdifferentiability.
\begin{defn}\label{ssddef}
    The norm of a Banach space $X$ is strongly subdifferentiable at a point $u\in X$ if
\begin{equation}\label{ssdeq}
\lim_{t \rightarrow 0^{+}} \frac{\|u+t x\|-\|u\|}{t}= \operatorname{Max}\{\operatorname{Re} f(x) : f \in J_{X^*}(u)\}
\end{equation}
uniformly for $x \in B_X$, where $J_{X^*}(u)=\left\{f \in S_{X^*} : f(u)=\|u\|\right\}$.
\end{defn}
If this happens, we say that $u$ is an $SSD$ point of $X$ or simply $u$ is an $SSD$ point when the Banach space under consideration is already known. 
 It is evident that an element $u \in X$ is an $SSD$ point of $X$ if and only if $\frac{u}{\|u\|}$ is an $SSD$ point of $X$. Hence, in the subsequent discussion, we will consider strong subdifferentiability at points on the unit sphere of a Banach space. 

We need the following two definitions in the sequel.
\begin{defn}
    Let $X$ and $Y$ be Banach spaces. A bounded linear operator $T : X \rightarrow Y$ is said to attain its norm if there exists an element $x \in S_X$ such that
$$
\|Tx\|=\|T\|. 
$$
\end{defn}
\begin{defn}\label{maxim}
A maximizing sequence for $T \in {B}(X, Y)$ is a sequence $\left\{x_n\right\}$ in $X$ with $x_n\in S_X$ for all $n$ so that $\{\left\|Tx_n\right\|\}$ converges to $\|T\|$ as $n \rightarrow \infty$.    
\end{defn}

We know, from \cite{MR1216708}, an element $u$ is an $SSD$ point of $X$  if and only if the face
$J_{X^*}(u)$
is strongly exposed by $u$, in the sense that the distance $d\left(f_n, J_{X^*}(u)\right)$ tends to zero for any sequence $\{f_n\}$ in $B_{X^*}$ satisfying $f_n(u) \rightarrow 1$. In \cite{MR1911087}, it is observed that for the dual norm of a Banach space $X$ to be strongly subdifferentiable at a bounded linear functional $f\in S_{X^*}$, it is enough to work with the set $
J_{X}(f)=\left\{x \in B_X:\|f\|=f(x)=1\right\}$ instead of $J_{X^{**}}(f)$. We will now recall this result from \cite{MR1911087}.

\begin{thm}\cite[Proposition 2.2]{MR1911087}\label{indu}
Let $X$ be a Banach space and $f \in S_{X^*}$. Then the following are equivalent. 
\begin{enumerate}
    \item[$(i)$] The norm of $X^*$ is strongly subdifferentiable at $f$. 
    \item[$(ii)$] $J_{X}(f)\neq\emptyset$, and for every sequence $\{x_n\}$ in $B_X$ satisfying $f(x_n)\rightarrow 1$, there exists a subsequence  $\{x_{n_i}\}$ of $\{x_n\}$ such that $d(x_{n_i}, J_{X}(f))\rightarrow 0.$
    
\end{enumerate} 
\end{thm}

\begin{rmk}\label{rmkindu}
    It is easy to see that condition (ii) above is equivalent to the statement, $J_{X}(f)=\emptyset$, and for each $\varepsilon>0$ there exists $\delta >0$ such that whenever $x\in S_X$ satisfies $f(x)>1-\delta$, we have $d(x, J_{X}(f))<\varepsilon$. In the subsequent sections of this paper, we use these two equivalent conditions interchangeably. 
\end{rmk} 

Similarly, it is observed in \cite[Theorem 1]{MR1373537} that an analogous characterization holds for the norm of $C^*$-algebra, in particular for the operator norm on Hilbert spaces. We will now recall this result. 

\begin{thm}\cite[Theorem 1]{MR1373537}\label{hilb}
    Let $\mathcal{A}$ be a $C^*$-algebra and $a \in S_\mathcal{A}$. The following assertions are equivalent.\\
(i) The norm of $\mathcal{A}$ is strongly subdifferentiable at a.\\
(ii) 1 is an isolated point in the spectrum of $|a|$.
\end{thm} 

 Assume that in the above theorem the $C^*$-algebra under consideration is ${B}(H)$ and the element $a\in S_\mathcal{A}$ corresponds to an operator $T\in S_{{B}(H)}$. Then we have the following equivalent characterization of strong subdifferentiability of operator norm on Hilbert spaces in terms of maximizing sequences/norm-attaining set of $T$. 
\begin{thm}\label{hil}
    Let $H$ be a Hilbert space and $T\in S_{{B}(H)}$. Then the following are equivalent.
    \begin{enumerate}
        \item[$(i)$] The norm of ${B}(H)$ is strongly subdifferentiable at $T$.
        \item[$(ii)$] $M_T\neq \emptyset$  and for every maximizing sequence $\{x_n\}$ of $T$, there exists a subsequence $\{x_{n_i}\}$ of $\{x_n\}$ such that $d(x_{n_i}, M_T)\rightarrow 0$.
        \item[($iii)$] $T$ attains its norm and $\|P_T|_{{M_T}^\perp}\|<1$, where $P_T=(T^*T)^{1/2}.$
    \end{enumerate}
\end{thm}
\begin{proof}
     It is enough to show that $(ii)$ and $(iii)$ are equivalent, since $(i)\Leftrightarrow (iii)$ follows from \cite[Theorem 1, proof of $(ii)\Rightarrow (iii)$]{MR1373537}. 
     
     Assume that $(iii)$ holds and write $M=\overline{\spn}M_T$,
     recall that  
     $$M_T=\{x\in S_H: \|Tx\|=\|T\|=1\}.$$
     Let  $\{x_n\}$ in $S_H$ be a maximizing sequence for $T$, that is, $\|Tx_n\|\rightarrow 1$. 
     We can write 
     $$x_n = z_n+w_n,$$
     where $z_n\in M$ and $w_n\in M^\perp$ with 
     \begin{equation}\label{xyz1}
         1=\|x_n\|^2=\|z_n\|^2+\|w_n\|^2.
     \end{equation}
    We also have $\langle P_Tz_n, P_Tw_n\rangle = 0$. 
    Therefore,  $$\|P_Tx_n\|^2=\|P_Tz_n\|^2+\|P_Tw_n\|^2.$$
     Since  $P_T$ is a positive operator and $z_n\in \overline{\spn}M_T = \overline{\spn}M_{P_T}$, we have 
     \begin{equation}\label{zn}
         P_Tz_n =z_n.
     \end{equation} Together with equation (\ref{xyz1}) and the fact that $\|Tx\| =\|P_Tx\|$ for all $x\in H$, we get
      \begin{align}
       \lim \|P_Tw_n\|^2 & =1- \lim\|P_Tz_n\|^2\\\nonumber
       & = 1-\lim\|z_n\|^2\\\nonumber
       & =\|w_n\|^2.
     \end{align}
     Now, if $\lim\|w_n\| \neq 0$, then we have $\lim \frac{ \|P_Tw_n\|^2}{\|w_n\|^2} = 1$. Since $w_n\in M^\bot$, we will then have $\|P_T|_{M^\bot}\|=1$. 

     Observe that, $M^\bot =M_T^\bot$. Therefore, we will have $\|P_T|_{M_T^\bot}\|=1$. 
     But, by the assumption, $\|P_T|_{M^\perp}\|<1$. 
     Therefore, $\lim\|w_n\| = 0$. That is, $\lim \|x_n -z_n\| =0$ and $\|z_n\|\rightarrow 1$. 
     
     Take the sequence $\{y_n\}$ to be $y_n =\frac{z_n}{\|z_n\|}$. Then, by Equation (\ref{zn}), $y_n\in M_T$. Also, since $\|z_n\|\rightarrow 1$, we get 
     $$\lim\|x_n-y_n\|= \lim\|x_n-z_n\| = 0.$$  This implies $(ii)$. 

     Conversely assume that $(ii)$ holds. If possible, assume that $\|P_T|_{{M_T}^\perp}\|=\|T\|.$ Then there exists a sequence $\{w_n\} $ in $S_{M_T^\perp}$ such that $\lim\|Tw_n\|=\lim\|P_Tw_n\| = 1$. Since $w_n\in M_T^\perp$, we have $d(w_n, M_T)\geq\|w_n\|=1$ for all $n$,
     which contradicts $(ii)$.  
\end{proof}
From the condition $(ii)$ of Theorem \ref{indu} and Theorem \ref{hil}, it is evident that a shared property of maximizing sequence of both bounded linear functionals on Banach spaces and operators on a Hilbert space characterize the strong subdifferentiability of the norm of respective spaces.  

The equivalent characterization (condition $(ii)$) given in Theorem \ref{hil} does not rely on the orthogonality properties specific to Hilbert spaces. Consequently, this property can be verified for more general Banach spaces. So, it is worth considering whether this new reformulation characterizes the strong subdifferentiability of the norm of the space of bounded linear operators on a Banach space which is not necessarily a Hilbert space.

	In the remaining part of this section, we will prove that Theorem \ref{hil} can be extended to $\ell_p$ spaces, $1<p<\infty$. To prove the sufficient part, we need the following lemma, which is similar to \cite[Theorem 1]{MR0845870}.
 \begin{lem}\label{lemmabi}
      Let $X$ and $Y$ be Banach spaces. Let $B\in {B}({X\times Y})$ with $\|B\|=1$. Then the norm of $B(X\times Y)$ is strongly subdifferentiable at $B$ if the following condition holds. 
      For each $\varepsilon> 0$ there exists $\delta>0$ such that, whenever $(x,y)\in S_X\times S_Y$ satisfies $B(x,y)>1-\delta$, there exists $(x_0,y_0)\in S_X\times S_Y$ satisfying $B(x_0,y_0)=1$, $\|x-x_0\|<\varepsilon$ and $\|y-y_0\|<\varepsilon$ 
 \end{lem}
 \begin{proof}
      The proof of this lemma utilizes a similar idea employed in the proof of \cite[Theorem 1]{MR0845870}.

     Suppose $B\in {B}({X\times Y})$ with $\|B\|=1$ and the given condition holds. Assume that the norm of $B(X\times Y)$ is not strongly subdifferentiable at $B$. Then, by the definition of strong subdifferentiability, there exists $\varepsilon>0$ and  bounded bilinear functionals $C_n$ on $X\times Y$ with $\|C_n\|< \frac{1}{n}$ such that 
     \begin{equation}\label{con}
	\left|\left\|B+C_n\right\|-\left\|B\right\|-\operatorname{Re}\phi\left(C_n\right)\right| \geq \varepsilon\left\|C_n\right\| \text { for all }  \phi\in J_{{B}(X\times Y)^*}(B).
	\end{equation}
 Using the assumed condition, choose a number $\delta>0$ which satisfies the following. 
 
 If $\operatorname{Re} B(x, y)> 1-\delta$ for some $(x, y)\in S_X\times S_{Y}$, then there exists $(x_0, y_0)\in S_X\times S_{Y}$ satisfying 
 \begin{enumerate}
     \item[$(i)$] $B(x_0, y_0) =1$,
     \item[$(ii)$] $\|x-x_0\|<\frac{\varepsilon}{4}$,  $\|y-y_0\|<\frac{\varepsilon}{4}.$
     
 \end{enumerate}
We can assume, w. l. o. g,  that $\delta<\frac{\varepsilon}{2}$. From Equation (\ref{con}), we can find $C_n$ such that  $\|C_n\|<\frac{\delta}{2+\delta}$ which satisfies the Inequality (\ref{con}). Let $C = C_n$. 
Choose an element $\left(x, y\right) \in B_X \times B_{Y}$ such that
$$\left(B+C\right)\left(x, y\right)>\left\|B+C\right\|-\delta\left\|C\right\|.$$
Then we have:
\begin{align*}
\left\|B\right\| \geq \operatorname{Re} B\left(x, y\right) & =\left(B+C\right)\left(x, y\right)-\operatorname{Re} C\left(x, y\right) \geq\left(B+C\right)\left(x, y\right)-\left\|C\right\| \\
& \geq\left\|B+C\right\|-(1+\delta)\left\|C\right\|\\
&\geq\|B\|-(2+\delta)\left\|C\right\|\\
&> 1-\delta.
\end{align*}
Hence, $\operatorname{Re} B\left(x, y\right) >1-\delta$. By choice of the number $\delta$, there exists  $(x_0, y_0)\in S_X\times S_{Y}$ such that $B(x_0, y_0) =1$  and $\|x-x_0\|<\frac{\varepsilon}{4}$ and $\|y-y_0\|<\frac{\varepsilon}{4}$. 

Recall the isometric isomorphism $B(X\times Y) \cong(X \tilde{\otimes}_\pi Y)^*$. Let $c$ be the bounded linear functional on $X \tilde{\otimes}_\pi Y$ corresponding to the bilinear functional $C$. Then
\begin{align}\label{wer}
\left|C\left(x, y\right)-C\left(x_0, y_0\right)\right| & =\left|c\left(x \otimes y-x_0 \otimes y_0\right)\right| \\\nonumber
& \leq\left\|C\right\|\left\|x \otimes y-x_0 \otimes y_0\right\|_\pi\\\nonumber
& \leq \left\|C\right\|\left\|(x-x_0)\otimes y + x_0\otimes (y-y_0)\right\|_\pi\\\nonumber
& \leq \|C\| \frac{\varepsilon}{2}.
\end{align}
We also have:
\begin{equation}\label{qw}
    \left\|B+C\right\| \geq\operatorname{Re} \left(B+C\right)\left(x_0, y_0\right)=\left\|B\right\|+\operatorname{Re} C\left(x_0, y_0\right),
\end{equation}
From the Inequalities (\ref{wer}) and (\ref{qw}), we have
$$
\begin{aligned}
0 & \leq\left\|B+C\right\|-\left\|B\right\|-\operatorname{Re} C\left(x_0, y_0\right) \\
& <\left(B+C\right)\left(x, y\right)+\delta\left\|C\right\|-\left\|B\right\|-\operatorname{Re} C\left(x_{0}, y_0\right) \\
& <\delta\left\|C\right\|+\operatorname{Re} C(x, y)-\operatorname{Re} C(x_0, y_0) \\
&\leq \delta\left\|C\right\|+ \left|C\left(x, y\right)-C\left(x_0, y_0\right)\right| \\
& < \left\|C\right\|(\delta+\frac{\varepsilon}{2})\\
&< \|C\|\varepsilon,
\end{aligned}
$$
which is a contradiction to the Inequality (\ref{con}), since $(x_0, y_0)\in J_{{B}(X\times Y)^*}(B)$. Therefore, the norm of $B(X\times Y)$ is strongly subdifferentiable at $B$.
 \end{proof}
 We will next recall a theorem from \cite{MR3176146}.

 \begin{thm}\cite[Corollary 2.2]{MR3176146}\label{unif}
     A reflexive Banach space $X$ is uniformly smooth if and only if for every $\varepsilon>0$ there is $0<\eta(\varepsilon)<1$ such that, for all $f \in B_{X^*}$ and all $x \in S_X$ satisfying $|f(x)|>1-\eta(\varepsilon)$, there exists $f_0 \in S_{X^*}$ satisfying $\left|f_0(x)\right|=1$ and $\left\|f-f_0\right\|<\varepsilon$.
 \end{thm} 
 We will now give a sufficient condition for strong subdifferentiability of operator norm on Banach spaces when the range is uniformly smooth.
 
	\begin{prop}\label{uni}
	    Let $X$ and $Y$ be Banach spaces and $T\in S_{{B}(X, Y)}$. Then the norm of ${B}(X, Y)$ is strongly subdifferentiable at $T$ if the following conditions hold.
     \begin{enumerate}
     \item[$(i)$] $M_T\neq \emptyset$  and for every maximizing sequence $\{x_n\}$ of $T$, there exists a subsequence $\{x_{n_i}\}$ of $\{x_n\}$ such that $d(x_{n_i}, M_T)\rightarrow 0$.
     \item[$(ii)$] $Y$ is uniformly smooth. 
          \end{enumerate}
	\end{prop}
 \begin{proof}
	Assume that an operator $T\in S_{{B}(X, Y)}$ and the conditions $(i), (ii)$ hold. Since $Y$ is uniformly smooth, we have the isometric isomorphism between the following Banach spaces.
 $${B}(X, Y)\cong{B}(X, Y^{**}) \cong(X\hat{\otimes}_\pi Y^*)^* \cong {B}(X\times Y^*).$$
 
 So it is enough to prove that the norm of ${B}(X\times Y^*)$ is strongly subdifferentiable at $T$ when viewed as a bounded bilinear functional on $(X\times Y^*)$. Using Lemma \ref{lemmabi}, it is enough to prove the following property (P).  
 
 (P) : For each $\varepsilon>0$, there exists $\delta>0$ such that, whenever $(x, y^*)\in S_{X}\times S_{Y^*}$ satisfies $T(x, y^*)>1-\delta$, there exists $x_0\otimes y_0^*\in J_{B(X\times Y^*)^*}(T)$ satisfying $\|x-x_0\|<{\varepsilon}$ and $\|y-y_0^*\|<{\varepsilon}$. 

To see this, let $\varepsilon > 0$. Since $Y$ is uniformly smooth, by Theorem \ref{unif}, we have the following property.

(P1) : We can find $\delta_1> 0$ such that whenever an element $(y, y^*)\in S_Y\times S_{Y^*}$ satisfies $|y^*(y)|> 1-\delta_1$, there exists an element $y_0^*\in J_{Y^*}{(y)}$ with $\|y^*-y_0^*\|<\frac{\varepsilon}{2}$.

From $(i)$, we also have the following. 

(P2) : There exists $\delta_2>0$ such that whenever an element $x\in B_X$ satisfies $\|Tx\|>1-\delta_2$, we have $d(x, M_T)<\frac{\delta_1}{2}$.

Define $$\delta = \min\left\{\frac{\delta_1}{2},\delta_2,\varepsilon\right\}.$$ 

Choose an element $(x, y^*)\in S_X\times S_{Y^*}$ such that $T(x, y^*)>1-\delta$.
Then 
$$\|Tx\|\geq T(x, y^*)> 1-\delta.$$
Then, by property (P2), there exists $x_0\in M_T$ such that 
\begin{equation*}
    \|x-x_0\|<\frac{\delta_1}{2}.
\end{equation*}
Therefore,
\begin{equation*}
    |y^*(Tx_0)|= |y^*(Tx_0-Tx+Tx)|> \operatorname{Re}y^*(Tx_0-Tx+Tx)>\frac{-\delta_1}{2}+1-\delta>1-\delta_1.
\end{equation*}
Since $Tx_0\in S_Y$, by property (P1), there exists $y_0^*\in J_{Y^*}{(Tx_0)}$ such that $\|y^*-y_0^*\|<{\varepsilon}.$ 

Thus, we have $(x_0, y_0^*)\in J_{{B}(X\times Y^*)^*}(T)$ which satisfies the requirements of property (P) above. Therefore, the norm of ${B}(X\times Y^*) $ is strongly subdifferentiable at $T$. 
\end{proof}
We are now in a position to give the main theorem of this section. We will next obtain a characterization of strong subdifferentiability of the norm of bounded linear operators on $\ell_p$ spaces, analogous to that of operators on Hilbert space. Towards this, we recall a few definitions and results. 
\begin{defn}
     Let $X$ be a Banach space. For $f\in S_{X^*}$ and $\delta>0$, the slice of $B_X$ corresponding to $f$ and $\delta$ is defined to be  $S(f, \delta, B_X) = \{x\in B_X:\operatorname{Re}f(x)>1-\delta\}$.
\end{defn}
\begin{rmk}\label{slice}
    It is easy to see that if $X$ is a uniformly convex Banach space, then for an $f\in S_{X^*}$ and $\varepsilon>0$ we have  $\operatorname{diam}S(f, \delta(\varepsilon), B_X)<\varepsilon$, where $\delta(\varepsilon)$ is the modulus of convexity of $X$ corresponding to $\varepsilon$.
\end{rmk}
We will now recall the definition of $M$-ideals in Banach spaces.
\begin{defn}
    A closed subspace $J$ of a Banach space $X$ is called an $M$-ideal if there exists a closed subspace $J'$ of $X^*$ such that $X^*=J^\bot\oplus_1 J'$, where $J^\bot = \{f\in X^*: f|_J =0\}$.
\end{defn}
We need the following property of $M$-ideals in the proof of our main theorem. 
\begin{lem}\cite{MR1238713}\label{extrememideal}
    Let $J$ be an $M$-ideal in a Banach space $X$. Then 
    $$\operatorname{ext} B_{X^*} = \operatorname{ext} B_{J^\bot}\cup \operatorname{ext} B_{J^*}.$$
\end{lem}
We will next recall a theorem that characterizes the extreme points in the duals of operator spaces.
\begin{thm}\cite{MR0682665}\label{extreruess}
     Let $X$ and $Y$ be Banach spaces. Then we have:  $$\operatorname{ext} B_{{K}(X, Y)^*} =\operatorname{ext} B_{X^{**}}\otimes \operatorname{ext} B_{Y^*}.$$ We may replace ${K}(X, Y)$ by any linear subspace containing $X^* \otimes Y$.
\end{thm}
The following is the main theorem of this section, in which we give an affirmative answer, in the case of $\ell_p$ space, to Question \ref{qn1} proposed in the introduction. This theorem has many interesting applications including the denseness of strongly subdifferentiable operators and a characterization of Fr\'{e}chet differentiability, both of which are discussed in Section 3.
\begin{thm}
     \label{main1}
     Let $1< p, q<\infty$ and an operator $T\in {B}(\ell_p, \ell_q)$   with $\|T\|=1$. Then the following are equivalent. 
     \begin{enumerate}
    \item[$(i)$] Norm of ${B}(\ell_p, \ell_q)$ is strongly subdifferentiable at $T$.
		\item[$(ii)$] $M_T\neq \emptyset$  and for every maximizing sequence $\{x_n\}$ of $T$, there exists a subsequence $\{x_{n_i}\}$ of $\{x_n\}$ such that $d(x_{n_i}, M_T)\rightarrow 0$.
  \end{enumerate}
\end{thm}     
\begin{proof} The implication $(ii)\Rightarrow(i)$ follows from Proposition \ref{uni}. 

We will now prove the implication $(i)\Rightarrow (ii)$. For, it is enough to prove that, for each $\varepsilon>0$, there exists $\delta >0$ such that
\begin{equation}\label{nec}
    d(x, M_T)\leq \varepsilon \mbox{ whenever } \|T(x)\|>1-\delta \mbox{ for some } x\in B_{\ell_p}.
\end{equation} 
To begin, assume that the norm of ${B}(\ell_p, \ell_q)$ is strongly subdifferentiable at an operator $T\in S_{{B}(\ell_p, \ell_q)}$. 
Observe that, we have the following isometric identification,
 $${B}(\ell_p, \ell_q) \cong (\ell_p\hat{\otimes}_\pi \ell_q^*)^*.$$
 under the map 
 \begin{equation}\label{tilde}
     {B}(\ell_p, \ell_q)\ni T \rightarrow \tilde{T}\in (\ell_p\hat{\otimes}_\pi \ell_q^*)^*,
 \end{equation}
 where the action of $\tilde{T}$ on $\ell_p\hat{\otimes}_\pi \ell_q^*$ is given by 
$$\tilde{T}\left(\sum_{i=1}^nx_i\otimes y_i^*\right) = \sum_{i=1}^ny_i^*(Tx_i). $$
Therefore, the norm of the dual space $(\ell_p\hat{\otimes}_\pi \ell_q^*)^*$ is strongly subdifferentiable at $\tilde{T}$.

Fix $\varepsilon >0$ and let $\delta_1>0$ be the modulus convexity of the space $\ell_p$ corresponding to $\varepsilon$. 
Apply Remark \ref{rmkindu} to obtain $\delta >0 $ such that 
\begin{equation}\label{pr}
    d(v,J_{\ell_p\hat{\otimes}_\pi \ell_q^*}(\tilde{T}))< \delta_1 \mbox{ whenever } \tilde{T}(v) >1-\delta \mbox{ for some }  v\in B_{\ell_p\hat{\otimes}_\pi \ell_q^*}.
\end{equation}
We will now prove that $T$ satisfies the requirements of (\ref{nec}) corresponding to the $\delta>0$ found in (\ref{pr}). For, let $x\in S_{\ell_p}$ such that $\|Tx\|>1-\delta$. Then there exists a bounded linear functional $y^*\in S_{\ell_q^*}$ such that $ y^*(Tx) = \|Tx\|$. But
$$\tilde{T}(x\otimes y^*) = y^*(Tx) = \|Tx\| >1-\delta. $$
Therefore, from the expression (\ref{pr}), it follows that there exists an element $u\in J_{\ell_p\hat{\otimes}_\pi \ell_q^*}(\tilde{T})$ such that 
$$\|x\otimes y^*-u\|_\pi<\delta_1.$$
At this stage, we will try to get a representation of $u$. For, we will look at the set $J_{{B}(\ell_p, \ell_q)^*}(T)$. It is clear that 
$$u\in J_{\ell_p\hat{\otimes}_\pi \ell_q^*}(\tilde{T})\subseteq J_{{B}(\ell_p, \ell_q)^*}(T).$$
Observe that $J_{{B}(\ell_p, \ell_q)^*}(T)$ is a $w^*$-closed extremal subset of the closed unit ball of $ {{B}(\ell_p, \ell_q)^*}$, where a subset $E$ of a convex set $C$ is called extremal if $x,y\in E$ whenever $x,y\in C$ and $tx+(1-t)y\in E$, $t\in (0,1)$.  

Hence, $J_{{B}(\ell_p, \ell_q)^*}(T)$ is the $w^*$-closed convex hull of its extreme points. In particular, 
   $$u\in J_{{B}(\ell_p, \ell_q)^*}(T) =  \overline{\conv}^{w^*}(\operatorname{ext}J_{{B}(\ell_p, \ell_q)^*}(T)).$$ 
   That is,
   $$ u=w^*-\lim_{k \rightarrow \infty} \sum_n^{m_k} t_n^k \varphi_n^k,$$ where $ \sum_n^{m_k} t_n^k =1$ for each $k$, and $\varphi_n^k\in \operatorname{ext}J_{{B}(\ell_p, \ell_q)^*}(T)$ for all $n,k$, consequently, $\tilde{T}(\varphi_n^k) =1$. 
   
   Since $J_{{B}(\ell_p, \ell_q)^*}(T)$ is extremal, we have 
   \begin{equation}\label{lkj}
       \operatorname{ext}J_{{B}(\ell_p, \ell_q)^*}(T)\subseteq \operatorname{ext}B_{{B}(\ell_p, \ell_q)^*}.
   \end{equation}
     It is well known that ${K}(\ell_p, \ell_q)$ is an $M$-ideal in ${B}(\ell_p, \ell_q)$ \cite{MR1238713}. Therefore, by Lemma \ref{extrememideal}, we have
$$\operatorname{ext}B_{{B}(\ell_p, \ell_q)^*} = \operatorname{ext}B_{{K}(\ell_p, \ell_q)^\perp}\cup \operatorname{ext}B_{{K}(\ell_p, \ell_q)^*}.$$
From the elementary theory of projective tensor products, we know that 
$${K}(\ell_p, \ell_q)^* = \ell_p\hat{\otimes}_\pi\ell_q^*.$$ 
Hence
$$\operatorname{ext}B_{{B}(\ell_p, \ell_q)^*} = \operatorname{ext}B_{{K}(\ell_p, \ell_q)^\perp}\cup \operatorname{ext}B_{\ell_p\hat{\otimes}_\pi\ell_q^*}.$$
From Equation (\ref{lkj}), we conclude that 
$$\operatorname{ext}J_{{B}(\ell_p, \ell_q)^*}(T)\subset\operatorname{ext}B_{{K}(\ell_p, \ell_q)^\perp}\cup \operatorname{ext}B_{\ell_p\hat{\otimes}_\pi\ell_q^*}.$$
Therefore, either $\varphi_n^k\in \operatorname{ext}B_{{K}(\ell_p, \ell_q)^\perp}$ or $\varphi_n^k\in \operatorname{ext}B_{\ell_p\hat{\otimes}_\pi\ell_q^*}$ for each $n$ and $k$. 
Define the set, 
$$L^k=\{n \in \mathbb{N}: \varphi_n^k\in \operatorname{ext}B_{{K}(\ell_p, \ell_q)^\perp}\}.$$

If $L_k=\{1, \ldots, m_k\}$ for all $k$, then 
$$z^k =\sum_{n\in L^k}t_n^k\varphi_n^k = \sum_{n=1}^{m_k}t_n^k\varphi_n^k\in B_{{K}(\ell_p, \ell_q)^\perp}$$ for all $k$. This would imply $u\in B_{{K}(\ell_p, \ell_q)^\perp}$, since ${K}(\ell_p, \ell_q)^\perp$ is weak* closed. But $u\in J_{\ell_p\hat{\otimes}_\pi \ell_q^*}(\tilde{T})$, forcing $u=0$, a contradiction. 

Therefore, for each $k$, there exists $\varphi_n^k\in \operatorname{ext}B_{\ell_p\hat{\otimes}_\pi\ell_q^*}$ for some $n\in\{1,\ldots, m_k\}$.
From Theorem \ref{extreruess}, we know that 
$$\operatorname{ext}B_{(\ell_p\hat{\otimes}_\pi \ell_q^*)} =\operatorname{ext}B_{{K}(\ell_p, \ell_q)^*} = \{x\otimes y: x\in S_{\ell_p}, y\in S_{\ell_q^*}\}.$$
Now, for each $n\notin L^k$, we have $\varphi_n^k\in \operatorname{ext}B_{\ell_p\hat{\otimes}_\pi\ell_q^*}$ and hence
\begin{equation}\label{unk}
    \varphi_n^k=  u_n^k \otimes v_n^{*k}, \mbox{ where } u_n^k\in S_{\ell_p}, v_n^{*k}\in S_{\ell_q^*}.
\end{equation}

Now, fix $\varphi_{n_0}^{k_0}\in \operatorname{ext}B_{\ell_p\hat{\otimes}_\pi\ell_q^*}\cap J_{{B}(\ell_p, \ell_q)^*}(T)\neq \emptyset.$ 
Since $\tilde{T}(\varphi_{n_0}^{k_0}) =1$, we have $v_{n_0}^{*k_0}(T u_{n_0}^{k_0}) =1$. Hence $\|Tu_{n_0}^{k_0}\|=1$ and thus $M_T\neq \emptyset$.  

 Let elements $x_0^*\in S_{{\ell_p}^*}$ and $y_0\in S_{\ell_q}$ be such that $x_0^*(x) =1$ and 
 $y^*(y_0)=1$, recall from the beginning that $y^*(Tx)=1-\delta$. Then $x_0^*\otimes y_0$ defines a compact operator from $\ell_p$ to $\ell_q$ or in other words $x_0^*\otimes y_0\in \ell_{p}^*\check{\otimes}_\varepsilon \ell_q$. Therefore, we can find $k$ such that, since $\lim_kx_0^*\otimes y_0(z^k) = 0,$ 
 $$\left|x_0^*\otimes y_0\left(u-\sum_{n\notin L^k} t_n^k u_n^k \otimes v_n^{*k} \right)\right|<{\delta_1-\|x\otimes y^*-u\|_\pi}.$$

 Thus for the above $k$, we have
\begin{equation}\label{impo}
\left|x_0^*\otimes y_0\left(x\otimes y^*-\sum_{n\notin L^k} t_n^k u_n^k \otimes v_n^{*k}\right)\right|<{\delta_1}.
\end{equation}
It remains to prove that $d(x, M_T)<\varepsilon$. 
If possible, suppose that $d(x, M_T)>\varepsilon$. 

Recall  that the elements $x_0^*\in S_{\ell_p^*}$ and $y_0\in S_{\ell_q}$ satisfies $x_0^*(x) =1$ and 
 $y^*(y_0)=1$. Then, by Remark \ref{slice}, we get 
 $$\operatorname{diam}S(x_0^*, \delta_1, B_{\ell_p})<\varepsilon.$$ 
 Since we are assuming that $d(x, M_T)>\varepsilon$ and $x\in S(x_0^*, \delta_1, B_{\ell_p})$, we can conclude that 
 $$y\notin S(x_0^*, \delta_1, B_X) \quad\mbox{ for all } y\in M_T. $$ 
 Therefore, by the definition of the slice. 
 $$\operatorname{Re}x_0^*(y)<1-\delta_1 \quad\mbox{ for all }y\in M_T. $$
Since $M_T = -M_T$, we have  
  \begin{equation}\label{slci}
      |\operatorname{Re}x_0^*(y)|<1-\delta_1 \quad\mbox{ for all } y\in M_T.
  \end{equation}
   We can assume, w. l. o. g., the elements  $u_n^k$ in Equation (\ref{unk}) satisfies $x_0^*(u_n^k)\in \mathbb{R}$. To see this, note that, if  $x_0^*(u_n^k) = r_n^ke^{i\theta_n^k}$, where $0\leq r_n^k\leq 1$ and $0\leq \theta_n^k<  2 \pi$, then take $\varphi_n^k = e^{-i\theta_n^k}u_n^k\otimes e^{i\theta_n^k}v_n^{*k}$. Also, note that $\sum_{n\notin L^k} t_n^k\leq 1$. 
  
 So, for the same $k$ above, since the elements  $u_n^k$ are in $M_T$ for all $n\notin L^k$, 
 \begin{align*}
     |\operatorname{Re}x_0^*\otimes y_0(u_n^k \otimes v_n^{*k})| & =
           |\operatorname{Re}x_0^*(u_n^k)v_n^{*k}(y_0)|\\
     & <   1-\delta_1.
 \end{align*}
 Therefore,
	\begin{align*}
 \left|\operatorname{Re}x_0^*\otimes y_0\left(x\otimes y^*-\sum_{n\notin L^k} t_n^k u_n^k \otimes v_n^{*k}\right)\right|
 &=\left|\operatorname{Re}x_0^*(x)y^*(y_0)- \sum_{n\notin L^k} t_n^k \operatorname{Re} x_0^*\otimes y_0(u_n^k \otimes v_n^{*k})\right|\\
		&\geq 1-\left| \sum_{n\notin L^k} t_n^k\operatorname{Re}x_0^*(u_n^k)v_n^{*k}(y_0)\right|\\
		&\geq 1-\sum_{n\notin L^k}t_n^k\left|\operatorname{Re}x_0^*(u_n^k)v_n^{*k}(y_0)\right|\\ 
		&\geq 1-\sum_{n\notin L^k}t_n^k(1-\delta_1)\\
		&\geq\delta_1.
	\end{align*}
 Which is a contradiction to the Inequality (\ref{impo}). Hence $d(x, M_T)<\varepsilon$.
\end{proof}
\begin{rmk}
    In the proof of Theorem \ref{main1}, we have utilized the uniform convexity and uniform smoothness of $\ell_p$ spaces. By employing analogous arguments to those used in the proof of Theorem \ref{main1}, it can be shown that if Banach spaces $X$ and $Y$ satisfies the following properties: $X$ is uniformly convex, $Y$ is uniformly smooth, ${K}(X, Y)$ is an $M$-ideal in ${B}(X, Y)$, $X^*$ has Radon-Nikodym property, and either $X^*$ or $Y^*$ has approximation property, then the condition $(ii)$ of Theorem \ref{main1} characterizes the strong subdifferentiability of the norm of ${B}(X, Y)$.
\end{rmk}
The following is an example of an operator at which the norm of ${B}(\ell_p)$ is strongly subdifferentiable. 
  \begin{eg}
     Let $U\in{B}(\ell_p)$, $1< p<\infty$, be an isometry. Then $M_U = S_{\ell_p}$. Therefore, by Theorem \ref{main1}, $U$ is an $SSD$ point of ${B}(\ell_p)$.  
     
 \end{eg}
 We will next show that for certain Banach spaces $X$ and $Y$, condition $(ii)$ of Theorem \ref{main1} is necessary for the operator norm on ${B}(X, Y)$ to be strongly subdifferentiable at a bounded linear operator $T\in S_{{B}(X, Y)}$. Consequently, we will address the borderline case $p=\infty$ in Corollary \ref{ssd3}.    

Observe that for all Banach spaces $Y$, the pair $(Y, \ell_1)$ has the property that every element $u$ of  $Y \hat{\otimes}_\pi \ell_1 $ has a representation $u= \sum_{n=1}^\infty\lambda_n u_n\otimes v_n$, where $\lambda_n \geq 0$, $\sum_{n=1}^\infty\lambda_n=\|u\|$, $\|u_n\|=\|v_n\|=1$ \cite{MR1888309}.

Motivated by the above property of $\ell_1$, we define the following for our current purpose.
\begin{defn}
    A pair of Banach spaces $(X, Y)$ is said to have property $(N)$ if the following hold.
    Whenever $f\in (X\hat{\otimes}_\pi Y)^*$ is an $SSD$ point and satisfies $f(u)=\|f\|$ for some $u\in B_{X\hat{\otimes}_\pi Y}$, then $u$ has the form,  $u= \sum_{i=1}^\infty\lambda_i u_i\otimes v_i$ with $\|u_i\|=1$, $\|v_i\|=1$, and $\sum_{i=1}^\infty\lambda_i =1$.
\end{defn}

Examples of Banach spaces satisfying the property $(N)$ include  $(Y, \ell_1)$ for every Banach space $Y$, the pair $(H, H)$ where $H$ is a Hilbert space, and the pair $(X, Y)$ for finite-dimensional Banach spaces $X$ and $Y$. Besides these examples, we do not know a pair of Banach spaces satisfying property $(N)$.  We refer the reader \cite{MR4371175}, where a stronger notion of property $(N)$ has been studied. 

We will next prove a result wherein the condition $(ii)$ of Theorem \ref{main1} is necessary for the operator norm to be strongly subdifferentiable. 
\begin{thm}\label{ssd1}
	Let $X$ and $Y$ be Banach spaces such that $X$ is uniformly convex and the pair $(X, Y^*)$ has property $(N)$. If the norm of ${B}(X, Y^{**})$ is strongly subdifferentiable at $T\in S_{{B}(X, Y^{**})}$, then 
	 $M_T\neq \emptyset$ and for every maximizing sequence $\{x_n\}$ of $T$, there exists a subsequence $\{x_{n_i}\}$ of $\{x_n\}$ such that $d(x_{n_i}, M_T)\rightarrow 0$.
\end{thm}
\begin{proof} 
	  Suppose the norm of ${B}(X, Y^{**})$ is strongly subdifferentiable at $T$. We have $(X\hat{\otimes}_\pi Y^*)^* \cong {B}(X, Y^{**})$. Hence the norm of $(X\hat{\otimes}_\pi Y^*)^*$  is strongly subdifferentiable at $\tilde{T}$ (see Expression (\ref{tilde}) in Theorem \ref{main1}).   
	Let $\varepsilon >0$ and $\delta_1$ be the modulus convexity of $X$ corresponding to $\varepsilon$. Using Theorem \ref{indu}, choose $\delta >0 $ such that whenever an element $u\in B_{X\hat{\otimes}_\pi Y^*}$ satisfies $\tilde{T}(u) >1-\delta $, we have $d(u,J_{X\hat{\otimes}_\pi Y^*}(T))< \delta_1$.
 
	Let $x\in S_X$ such that $1-\delta<\|Tx\|$. Then there exists a bounded linear functional $y^*\in S_{Y^*}$ such that $ y^*(Tx) = \|Tx\|$. 

	Since the element $x\otimes y^*\in B_{X\hat{\otimes}_\pi Y^*} $ satisfies $\tilde{T}(x\otimes y^*)>1-\delta $, there exists $u\in J_{X\hat{\otimes}_\pi Y^*}(T)$ such that $\|x\otimes y^*-u\|_\pi<\delta_1$. 
	
 By our assumption, the pair has $(X, Y^*)$ has the property $(N)$. Therefore, 
 we can write \begin{equation*}
     u =\sum_{i=1}^\infty\lambda_ix_i\otimes y_i^*, \mbox{ where } \sum_{i=1}^\infty\lambda_i =1, \quad x_i\in S_X,\mbox{ and } y_i^*\in S_{Y^*}.
 \end{equation*} 
	Since $\tilde{T}(u) = 1$, we get $\tilde{T}(x_i\otimes y_i^*) = 1$. Consequently, $\|T(x_i)\|= \|y_i^*\|= 1$ for all $i$ and thus $M_T\neq \emptyset$. Moreover,
	\begin{equation}\label{pi}
		\left\|x\otimes y^*-\sum_{i=1}^\infty\lambda_ix_i\otimes y_i^*\right\|_{X\hat{\otimes}_\pi Y^*}<\delta_1.
	\end{equation}
We will now prove that $d(x, M_T)<\varepsilon$. If possible, suppose that $d(x, M_T)>\varepsilon$. Let the element $x_0^*\in S_{X^*}$ be such that $x_0^*(x) = 1$. An argument,  as in the proof of Theorem \ref{main1}, using the slice of $B_X$ corresponding to the above $\delta>0$ and $x_0^*$ will obtain that
 \begin{equation}\label{del}
     \left\|x_0^*(x)y^*-\sum_{i=1}^\infty\lambda_{i}\operatorname{Re}x_0^*(x_i)y_i^*\right\|_{Y^*} > \delta_1.
 \end{equation}
 Since the projective tensor norm is larger than the injective tensor norm, we get, using Inequality (\ref{del})
 \begin{align*}
		 \left\|x\otimes y^*-\sum_{i=1}^\infty\lambda_ix_i\otimes y_i^*\right\|_{X\hat{\otimes}_\pi Y^*}&\geq \left\|x\otimes y^*-\sum_{i=1}^\infty\lambda_ix_i\otimes y_i^*\right\|_{X\check{\otimes}_\varepsilon Y^*}\\
       &=\left\| x\otimes y^*-\sum_{i=1}^\infty\lambda_ix_i\otimes y_i^*\right\|_{{B}(X^*, Y^*)}\\
		&\geq \left\|x_0^*(x)y^*-\sum_{i=1}^\infty\lambda_i\operatorname{Re}x_0^*(x_i)y_i^*\right\|_{Y^*} \\
		&> \delta_1.
	\end{align*}
 Which is a contradiction to Inequality (\ref{pi}). Hence  $d(x, M_T)<\varepsilon$.
 \end{proof}
In the following corollary, we will give a necessary condition for the operator norm to be strongly subdifferentiable in the borderline case $p=\infty$.
\begin{cor}\label{ssd3}
Let $X$ be a uniformly convex Banach space and $T\in {B}({X, \ell_\infty})$ such that $\|T\|=1$. If the norm of  ${B}({X, \ell_\infty})$ is strongly subdifferentiable at $T$, then $M_T\neq \emptyset$  and for every maximizing sequence $\{x_n\}$ of $T$, there exists a subsequence $\{x_{n_i}\}$ of $\{x_n\}$ such that $d(x_{n_i}, M_T)\rightarrow 0$.
\end{cor}
\begin{proof}
We have ${B}({X, \ell_\infty}) = {B}(X, c_0^{**})$. Since the pair $(X, c_0^*)=(X, \ell_1)$ has property $(N)$, the result follows from Theorem \ref{ssd1}. 
\end{proof}
We will next consider the other extreme value $p=1$. The strong subdifferentiability of the operator norm $\ell_1$ can be deduced from \cite[Theorem 2.5]{MR1216708} and Theorem \ref{indu}, for completeness, we will recall it here. 
\begin{prop} 
Let $T\in {B(\ell_1)}$ be an operator such that $\|T\|=1$ and $\{e_j\}$ denotes the canonical basis of $\ell_1$. Then $T$ is an $SSD$ point if and only if the following two conditions hold.
\begin{enumerate}
    \item[$(i)$] The set, $\{j\in \mathbb{N} : \|Te_j\| =\|T\|\}\neq \emptyset$ and $\sup \{\|Te_j\| : \|Te_j\| <1\}<1.$
    \item[$(ii)$] $Te_j\in c_{00}$ whenever $\|Te_j\|=1$.
\end{enumerate}
\end{prop}

\begin{proof}
We have $B(\ell_1) = \left(\bigoplus_{n=1}^\infty \ell_1\right)_{\infty}$ under the isometric identification $T \mapsto (Te_1, Te_2, \ldots)$ with $\|T\| =\sup_j\{\|Te_j\|\}$.

Therefore, $T$ is an $SSD$ point $B(\ell_1)$ if and only if $(Te_1, Te_2, \ldots)$ is an $SSD$ point of $\left(\bigoplus_{n=1}^\infty \ell_1\right)_{\infty}$.

From \cite[Theorem 2.5]{MR1216708}, the element $(Te_1, Te_2, \ldots)$  is an $SSD$ point of $\left(\bigoplus_{n=1}^\infty \ell_1\right)_{\infty}$ if and only if $Te_j$ is an $SSD$ point of $\ell_1$ for all $j$ satisfying $\|Te_j\| = 1$ and the element $(\|Te_j\|)\in \ell_\infty$ is an $SSD$ point of $\ell_\infty$. 

Since  $\ell_1 =c_0^*$, it can be easily verified using Theorem \ref{indu} that $Te_j$ is an $SSD$ point of $\ell_1$ if and only if $Te_j\in c_{00}$. Since $\ell_\infty =\left(\oplus_{n=1}^\infty \mathbb{F}\right)_\infty$, using  \cite[Theorem 2.5]{MR1216708} again, we get  that $(\|Te_j\|)\in \ell_\infty$ is an $SSD$ point of $\ell_\infty$ if and only if $\{j\in \mathbb{N} : \|Te_j\| =1\}\neq \emptyset$ and $\sup \{\|Te_j\| : \|Te_j\| <1\}<1.$ 
\end{proof}
We will next present a proposition that may be of independent significance. 
 \begin{prop}
    Let $X$ be a uniformly convex Banach space and $T\in S_{{B}(X, Y)}$ for some Banach space $Y$. Let the sequence $\{x_n\}$ be a maximizing sequence of $T$. Then $\liminf d(x_n, {\conv}M_T)=0$ if and only if $\liminf d(x_n, M_T)=0$.
\end{prop} 
 \begin{proof}
 We will prove that if $\liminf d(x_n, {\conv}M_T)=0$, then $\liminf d(x_n, M_T)=0$. The other implication is straightforward.

 Suppose that  $\liminf d(x_n, {\conv}M_T)=0$, passing onto a subsequence if necessary, we assume that $\lim\|x_n-y_n\|=0$ for some sequence $\{y_n\}$ in ${\conv}M_T$. 

    For each $n$, the element $y_n$ can be written as $y_n = \sum_{m=1}^{k_n}\lambda_{m}^{n}x_{m}^n$, where $\sum_{m=1}^{k_n}\lambda_{m}^n=1$
    and $x_{m}^n\in M_T$.
    
    If $\lim d(y_{n_i}, M_T)\neq0$ for every subsequence  $\{y_{n_i}\}$ of $\{y_n\}$, then there exists $\varepsilon> 0$ such that $\|y_n-x\|>\varepsilon$ for all $x\in M_T$. 
    Denote by $f_x$, the bounded linear functional on $X^*$ such that $f_x(x) =\|x\|$ for each $x\in S_X$.  
    Let $\delta>0$ be the modulus of convexity of $X$ corresponding to $\varepsilon$. 
     Then, for each $n$ and $x\in M_T$, we have $y_n\notin S(f_x,\delta, B_X)$, the slice of $B_X$ corresponding to $f_x$ and $\delta$. 
   Therefore, 
   \begin{equation}\label{eqqq}
    \operatorname{Re}f_x(y_n)<1-\delta \mbox{ for all } n \mbox{ and } x\in M_T.   \end{equation}
   Since $\{\|y_n\|\}$ converges to $1$, we get that $\{f_{y_n}(y_n)\}$ converges to $1$.
    Hence we can assume that $f_{y_n}(y_n)>1-\varepsilon_n$ for each $n$ and for some sequence $(\varepsilon_n)$ converging to $0$. Since  $y_n = \sum_{m=1}^{k_n}\lambda_{m}^nx_{m}^n$,
    it follows that $\sum_{m=1}^{k_n}\lambda_{m}^nf_{y_n}(x_{m}^n)>1-\varepsilon_n$. 
    So there exists $m_n$ corresponding to each $n$ such that $f_{y_n}(x_{m_n}^n)>1-\varepsilon_n$. 
    Therefore, the sequence $\{f_{y_n}(x_{m_n}^n)\}$ converges to $1$ as $n$ tends to $\infty$.
    
 But $|f_{{x_{m_n}^n}}((x_{m_n}^n)- f_{{x_{m_n}^n}}((y_{m_n}^n)| \leq \|x_n-y_n\|\rightarrow 0$.  
   
   Hence $f_{{x_{m_n}^n}}(y_n)\rightarrow 1$, which is a contradiction to inequality (\ref{eqqq}). Therefore, there exits a subsequence $\{y_{n_i}$ of $\{y_n\}$ such that $\lim_id(y_{n_i}, M_T)=0$. Consequently, $\lim_id(x_{n_i}, M_T)=0$.
\end{proof}


\section{Denseness of strongly subdifferentiable points of ${B}(\ell_p, \ell_q)$}

   We know that the set of all $SSD$ points in ${B}(H)$ is dense in ${B}(H)$ \cite{MR1373537}. In this section, we will extend this result to the operator norm on $\ell_p$ spaces. Additionally, we will show that the set of $SSD$ points of ${B}(\ell_p, \ell_q)$ contains an important class of operators, namely, those operators with essential norm strictly less than its operator norm. Towards the end of this section, we prove that the Fr\'{e}chet and Gateaux differentiability coincide in ${B}(\ell_p, \ell_q)$. We will also show that the assumption made by F. Kittaneh and R. Younis in \cite{MR1073855} regarding the smoothness of operator norm characterizes Fr\'{e}chet differentiability in ${B}(\ell_p, \ell_q)$.
   
     For $1<p,q<\infty$, let $\operatorname{SSD}\{{B}(\ell_p, \ell_q)\}$ denotes the set of all operators in ${B}(\ell_p, \ell_q)$ at which the norm of ${B}(\ell_p, \ell_q)$ is strongly subdifferentiable. 

     We will next provide a theorem which has very interesting consequences.
\begin{thm}\label{den}
The following inclusion holds for $1< p, q<\infty$.
\begin{equation}\label{inclu}
    \left\{T \in {B}(\ell_p, \ell_q): \|T\|_e<\|T\|\right\}\subset \operatorname{SSD}\{{B}(\ell_p, \ell_q)\}.
\end{equation}
In particular, $\operatorname{SSD}\{{B}(\ell_p, \ell_q)\}$ is a dense subset of ${B}(\ell_p, \ell_q)$.
\end{thm}
To establish Theorem \ref{den}, we have to integrate Theorem \ref{main1} with a few findings from \cite{MR1238713} and Lemma \ref{wmp}.
We will first recall the required results.
\begin{lem}\cite[Proposition 4.8, Proof of (b)]{MR1238713}\label{nowhere}
    Let $1<p,q<\infty$. Then the set ${K}^\theta(\ell_p, \ell_q)=\left\{T \in {B}(\ell_p, \ell_q): \| T\|=\| T \|_e\right\}$ is nowhere dense in ${B}(\ell_p, \ell_q)$.
\end{lem}
We will next recall a formula for computing the essential norm of an operator. 
\begin{lem}\cite[Lemma 5]{MR1155717}\label{esse}
Suppose that ${K}(X, Y)$ is an $M$-ideal in ${B}(X, Y)$ and an operator $T\in {B}(X, Y)$. Then
$$
\|T\|_e=\max \left\{w(T), w^*(T)\right\},
$$
where 
$$
w(T):=\sup \left\{\limsup _\alpha\left\|T x_\alpha\right\| :\left\|x_\alpha\right\|=1, x_\alpha \xrightarrow{w} 0\right\}
$$
and, in the same way,
$$
w^*(T):=\sup \left\{\limsup _\alpha\left\|T^* x_\alpha\right\|: \| x_\alpha \|=1, x_\alpha \xrightarrow{w^*} 0\right\}
$$
and at least one of the involved suprema is achieved.
\end{lem}
We will now recall the definition of $(M_p)$-space.
\begin{defn}\cite{MR1121227}
    Let $1 \leq p <\infty$. We say that a Banach space $X$ has property $\left(M_p\right)$ or is an $\left(M_p\right)$-space if ${K}\left(X \oplus_p X\right)$ is an $M$-ideal of ${B}\left(X \oplus_p X\right)$.
\end{defn}
We refer the reader \cite{MR1121227, MR1238713} for more details on $(M_p)$-spaces. 

Note that the space $\ell_p$ belongs to the class of $(M_p)$-space for  $1 \leq p < \infty$.
From \cite{MR1121227}, we know that if $X$ is an $\left(M_p\right)$-space, then $X$ is reflexive and ${K}(X)$ is an $M$-ideal in ${B}(X)$. 
Moreover, from \cite[chapter VI, Theorem 6.6]{MR1238713} we get, if $X$ is an $(M_p)$-space, then for any weakly null sequence $\{u_n\}$ in $X$, the following equality holds
\begin{equation}\label{1}
        \lim\|u+u_n\|^p = \|u\|^p+\lim\|u_n\|^p \quad (u\in X).
    \end{equation}
The following lemma is an extension of \cite[Theorem 1, 2]{MR2540517} to the broader class of Banach spaces known as $(M_p)$-spaces, under the additional assumption that $\|T\|_e<\|T\|$.
\begin{lem}\label{wmp}
    Let $X$ and $Y$ be separable Banach spaces which belongs to the classes $(M_p)$ and $(M_q)$, respectively, $1<p,q<\infty$. Suppose $T\in {B}(X, Y)$ such that $\|T\|_e<\|T\|$ and the sequence $\{x_n\}$ in $B_X$ is a maximizing sequence for $T$. If a subsequence $\{x_{n_i}\}$ of $\{x_n\}$ converges weakly to $x\neq 0$, then $\|Tx\|=\|T\|$ and $\|x\|=1$.
\end{lem}
\begin{proof}
For $q< p$, we have ${K}(X, Y) = {B}(X, Y)$ \cite[chapter VI, Corollary 5.14]{MR1238713}. Therefore, the results hold since compact operators are weak-to-norm sequentially continuous. 

Therefore, we can assume that $1<p\leq q<\infty$. 
Let $T\in {B}(X, Y)$, we can take $\|T\|=1$. Let the sequence $\{x_n\}$ be a maximizing sequence for $T$ and $\{x_n\}$ converges weakly to $x\neq 0$ (we are denoting subsequence of $\{x_n\}$ by $\{x_n\}$ itself).  Since $X$ is $(M_p)$ and $Y$ is $(M_q)$, it follows from Equation (\ref{1}) that
    \begin{equation}\label{xn}
        1-\|x\|^p =\lim\|x_n-x\|^p
    \end{equation}
    and
     \begin{equation}\label{txn}
        1=\lim\|Tx_n\|^q =\|Tx\|^q+\lim\|Tx_n-Tx\|^q.
    \end{equation}  
    Note that $\|Tx\|\leq 1$ and $\lim\|Tx_n-Tx\|\leq 1$, consequently, since $q\geq p\geq 1$, 
 \begin{align*}
 1=\lim\left\|Tx_n\right\|^{q} &=\|Tx\|^q+\lim\|Tx_n-Tx\|^q\\
 &\leq \|Tx\|^p+\lim\|Tx_n-Tx\|^p\\
 &\leq \|Tx\|^p+\lim\|x_n-x\|^p\\
 &= \|Tx\|^p+ 1-\|x\|^p.
\end{align*}
Therefore, 
\begin{equation}\label{tisx}
    \|Tx\|= \|x\|.
\end{equation} Since $\|x\|\leq 1$, using Equations (\ref{xn}) and (\ref{txn}) together with the fact that $q\geq p> 1$, we get
\begin{equation}\label{1of2}
    1-\|x\|^q\geq 1-\|x\|^p.
\end{equation}
and
\begin{equation}\label{2of2}
    \lim \|x_n-x\|^q = \lim (\|x_n-x\|^p)^{q/p} = (1-\|x\|^p)^{q/p}\leq 1-\|x\|^p. 
\end{equation}
If possible, assume that the sequence $\{\|x_n-x\|\}$ does not converge to $0$. Then there exists $\varepsilon>0$ and a subsequence $\{x_{n_k}\}$ of $\{x_n\}$ such that $\|x_{n_k}-x\| >\varepsilon$ for all $k$. Then, by Equations (\ref{1of2}) and (\ref{2of2}),

\begin{align*}
   \lim \left\|T\left(\frac{x_{n_k}-x}{\|x_{n_k}-x\|}\right)\right\|^q & =\frac{1-\|Tx\|^q}{\lim\|x_{n_k}-x\|^q}\\
    &=\frac{1-\|x\|^q}{\lim\|x_{n_k}-x\|^q}\\
    &\geq\frac{1-\|x\|^p}{1-\|x\|^p}\\
    &=1.
    \end{align*}
Since the sequence $\left\{\frac{x_{n_k}-x}{\|x_{n_k}-x\|}\right\}$ converges weakly to $0$, by Lemma \ref{esse}, 
\begin{equation*}
    \|T\|_e\geq  \lim \left\|T\left(\frac{x_{n_k}-x}{\|x_{n_k}-x\|}\right)\right\| \geq 1 =\|T\|. 
\end{equation*}
Which is not possible by our assumption. Therefore,  the sequence $(\|x_n-x\|)$  converges to $0$. Hence, by Equations (\ref{tisx}) and (\ref{xn}), $\|Tx\| =\|x\|=1 =\|T\|$.
\end{proof}
We will next give a proof of Theorem \ref{den}.

 \begin{proof}[Proof of Theorem \ref{den}] 

Let $T \in {B}(\ell_p, \ell_q) $ such that $\|T\|_e<\|T\|$. To see that the operator $T\in  \operatorname{SSD}\{{B}(\ell_p, \ell_q)\}$, it is enough to verify, by Theorem \ref{main1}, $\liminf d(x_n, M_T) = 0$ for every maximizing sequence $\{x_n\}$ of $T$. 

Let the sequence $\{x_n\}$ be a maximizing sequence for $T$, then there exists a subsequence of $\{x_n\}$ (we will denote the subsequence by $\{x_n\}$ itself) converges weakly to some $x\in \ell_p$ by the reflexivity of $\ell_p$. 

If $x=0$, then by Lemma \ref{esse}, $\|T\|=\|T\|_e$, which is not possible.

Therefore, $x\neq 0$, by Lemma \ref{wmp}, $\|Tx\| =\|T\|$, and $\|x\|=1$. Thus $x\in M_T$. 

Since the sequence $\{x_n\}$ is a maximizing sequence for $T$, from the proof of Lemma \ref{wmp}, the sequence $\{\|x_n-x\|\}$  converges to $0$. Thus $d(x_n, M_T)\rightarrow 0$. Hence, by Theorem \ref{main1}, the operator $T\in  \operatorname{SSD}\{{B}(\ell_p, \ell_q)\}$.

The denseness of the set $\operatorname{SSD}\{{B}(\ell_p, \ell_q)\}$ is now follows from Lemma \ref{nowhere}.
\end{proof}
We know that, for a Hilbert space $H$, the norm of ${B}(H)$ is strongly subdifferentiable at every compact operator on $H$ \cite{MR1397934}. We will next obtain a corollary which extends this result to $\ell_p$ spaces.
\begin{cor}
    Suppose $1<p,q<\infty$ and $K\in {B}(\ell_p, \ell_q)$ be a compact operator. Then the norm of ${B}(\ell_p, \ell_q)$ is strongly subdifferentiable at $K$.
\end{cor}
\begin{proof}
    The proof follows directly from the inclusion,  $\left\{T \in {B}(\ell_p, \ell_q): \|T\|_e<\|T\|\right\}\subseteq \operatorname{SSD}\{{B}(\ell_p, \ell_q)\}.$
\end{proof}
Observe from Theorem \ref{main1} that if an operator $T\in \operatorname{SSD}\{{B}(\ell_p, \ell_q)\}$, then $T$ attains its norm, the converse is not true as the following example suggest. 
\begin{eg}\label{proper}
    Let $1<p, q<\infty$ and $\{e_n\}$ denotes the canonical basis of $\ell_p$. Let $T\in {B}(\ell_p, \ell_q)$ be an operator defined as follows, for each $x=(\zeta_1, \zeta_2, \ldots)\in\ell_p$
    $$T\left(\sum_{i=1}^\infty \zeta_ne_n\right)= \zeta_1e_1+\sum_{i=2}^\infty \left(1-\frac{1}{n}\right)\zeta_ne_n.$$
Clearly,  $M_T=\{\alpha e_1 : \alpha\in \mathbb{F}, |\alpha|=1\}$. Let the sequence $\{x_n\}$ be a maximizing sequence for $T$. Then, for some subsequence $\{x_{n_i}\}$, $d(x_{n_i}, M_T)\rightarrow 0$ if and only if $x_{n_i}\rightarrow e_1$. But the sequence $\{e_n\}$ is a maximizing sequence for $T$ and does not contain any subsequence converging to $e_1$. Hence $(ii)$ of Theorem \ref{main1} fails to hold. Therefore, $T\notin \operatorname{SSD}\{{B}(\ell_p, \ell_q)\}$.
\end{eg}
\begin{rmk}
    It is well known that the set of all norm-attaining bounded linear operators on a reflexive Banach space $X$ is dense in $B(X)$ \cite{MR0160094}.  From Example \ref{proper} and Theorem \ref{main1} we get $\operatorname{SSD}\{{B}(\ell_p)\}$ is a proper subset of norm-attaining bounded linear operators on $\ell_p$. Moreover, from Theorem \ref{den}, $\operatorname{SSD}\{{B}(\ell_p)\}$ is dense in ${B}(\ell_p)$. 
\end{rmk}

As an application of Theorem \ref{den}, we will next obtain a characterization of Fr\'{e}chet differentiability. Consequently, we will show the equivalence between Fr\'{e}chet differentiability and Gateaux differentiability of the norm of ${B}(\ell_p, \ell_q)$ (see Theorem \ref{equiv}), just as observed in the case of operator norm on Hilbert spaces (see \cite[Theorem 3.1]{MR519008} and \cite[Theorem]{MR1149980}). We begin by recalling a few definitions and results.
\begin{defn}\label{fred}
    The norm of a  Banach space $X$ is Fr\'{e}chet differentiable at a point $u$ in the unit sphere $S_X$ if and only if there is a bounded linear functional $f \in X^*$ (unique) such that
\begin{equation}\label{fre}
    \lim_{t \rightarrow 0} \frac{\|u+t x\|-1}{t}=\operatorname{Re} f(x)
\end{equation}
uniformly on $ B_X$.
\end{defn} 
If we drop the uniformity assumption for $x$ in the above definition, we have the definition of Gateaux differentiability of the norm at $u$. That is, Gateaux differentiability demands only the existence of the limit in Equation (\ref{fre}) at each point $x\in B_X$.
When $x= u$, we see that the unique bounded linear functional $f$ in Definition \ref{fred} satisfies $\|f\|=f(u)=1$. Indeed, the functional $f$, which we call the gradient of the norm at $u$, is then uniquely determined by the condition
$$
\|f\|=f(u)=1.
$$
If the norm of a Banach space $X$ is Gateaux differentiable at a point $u\in S_X$, then it is usually said that $u$ is a smooth point of $S_X$. 

We will next recall a theorem from \cite{MR1151547}, which combined with Theorem \ref{den} of the present paper will give a characterization for Fr\'{e}chet differentiability in ${B}\left(\ell_p,\ell_q\right)$ in terms of the essential norm of an operator.
\begin{thm}\cite[Theorem 2.1]{MR1151547} and \cite[Theorem 1]{MR1155717}\label{gat}
Let $T \in {B}\left(\ell_p,\ell_q\right), 1<p,q<\infty$, and $\|T\|=1$. Then the following are equivalent:
\begin{enumerate}
\item[$(i)$] $T$ is smooth.
\item[$(ii)$] $\|T\|_e <1$ and if $\left\|T x_1\right\|=\left\|T x_2\right\|=1$ for $x_1, x_2 \in B_{\ell_p}$ then $x_1= \alpha x _2$, where $\alpha\in \mathbb{F}$ and $|\alpha| =1$.
\end{enumerate}
\end{thm}
In the following theorem, we will see that the condition $(ii)$ of Theorem  \ref{gat} characterizes Fr\'{e}chet differentiability rather than smoothness. Moreover, the following theorem is an extension of \cite[Theorem 3.1]{MR519008} and \cite[Theorem]{MR1149980}  to the class of $\ell_p$ spaces. At this stage, the proof of the following corollary looks simple, but Theorem \ref{den} and consequently Theorem \ref{main1} are very crucial in the proof. 
\begin{prop}\label{equiv}
     Let $T \in {B}\left(\ell_p,\ell_q\right), 1<p<\infty$, and $\|T\|=1$. Then the following are equivalent:
\begin{enumerate}
\item[$(i)$]  The norm of ${B}\left(\ell_p\right)$ is Fr\'{e}chet differentiable at $T$.
\item[$(ii)$] $\|T\|_e <1$ and if $\left\|T x_1\right\|=\left\|T x_2\right\|=1$ for $x_1, x_2 \in B_{\ell_p}$ then $x_1= \alpha x _2$, where $\alpha\in \mathbb{F}$ and $|\alpha| =1$.
\end{enumerate}
Consequently, the norm of ${B}(\ell_p, \ell_q)$ is Fr\'{e}chet differentiable at $T$ if and only if norm of ${B}(\ell_p, \ell_q)$ is Gateaux differentiable at $T$.
\end{prop}
\begin{proof}
    Fr\'{e}chet differentiability implies  Gateaux differentiability (smoothness). Hence, by Theorem \ref{gat},  $(i)\Rightarrow (ii)$.

    Conversely, if $(ii)$ holds, then by Theorem \ref{gat}, $T$ is smooth.
    Also, by Theorem \ref{den}, $T$ is an $SSD$ point of ${B}(\ell_p, \ell_q)$. Since Gateaux differentiability together with strong subdifferentiability implies Fr\'{e}chet differentiability, the norm of ${B}(\ell_p, \ell_q)$ is Fr\'{e}chet differentiable at $T$. 
\end{proof}
{\bf Acknowledgement:} A part of this research was completed during the author's tenure as a Visiting Scientist at ISI Bangalore and ISI Kolkata. The author extends his heartfelt gratitude to ISI Bangalore and ISI Kolkata, especially to Prof. Jaydeb Sarkar of ISI Bangalore and Prof. Debashish Goswami of ISI Kolkata, for their generous hospitality and unwavering support during the period. A portion of this research has been supported through the JC Bose National Fellowship and grant of Prof. Debashish Goswami, sponsored by the Science and Engineering Research Board (SERB) under the Government of India. 
\def\romsup#1{{\edef\next{\the\font}$^{\next#1}$}}

\bibliographystyle{amsplain}
\end{document}